\documentclass[a4paper, 10pt]{amsart}

\usepackage{amsmath, latexsym,  amssymb, amscd}
\usepackage[all]{xy}
\usepackage[british]{babel}
\usepackage{url}

\usepackage{enumitem}

\usepackage{hyperref}

\renewcommand{\epsilon}{\varepsilon}

\DeclareMathOperator{\codim}{codim}
\DeclareMathOperator{\cod}{codim}

\DeclareMathOperator{\emor}{End}

\newcommand{\abs}[1]{\left| #1 \right|}

\renewcommand{\P}{\mathbb{P}}

\newcommand{\sotto}{B}

\newcommand{\Mm}{\mathcal{M}}

\newcommand{\qe}{\mathbb{Q}}
\newcommand{\C}{\mathbb{C}}
\newcommand{\R}{\mathbb{R}}
\newcommand{\Z}{\mathbb{Z}}

\newcommand{\Q}{\mathbb{Q}}

\newcommand{\Ci}{\mathcal{C}}

\newcommand{\cuno}{c_1(N)}
\newcommand{\cdue}{c_2}
\newcommand{\cquindici}{c_{3}}
\newcommand{\ccinque}{c_4}
\newcommand{\ctre}{c_5}
\newcommand{\cquattro}{c_6}
\newcommand{\csei}{c_7}
\newcommand{\csette}{c_8}
\newcommand{\cotto}{c_9}
\newcommand{\cdieci}{c_{10}}
\newcommand{\cnove}{c_{11}}
\newcommand{\cundici}{c_{12}}

\newcommand{\csedici}{c_{3}}
\newcommand{\cdiciassette}{c_{16}}
\newcommand{\cdiciotto}{c_{17}}

\newtheorem{thm}{Theorem}[section]

\newtheorem*{con}{Conjecture}
\newtheorem{propo}[thm]{Proposition}
\newtheorem{lem}[thm]{Lemma}
\newtheorem{cor}[thm]{Corollary}
\newtheorem{D}[thm]{Definition}

\newtheorem{remark}[thm]{Remark}

\title[]{On the explicit torsion anomalous conjecture}
\author{S. Checcoli, F. Veneziano, E. Viada}
\begin{document}

\thanks{This paper has been accepted for publication on Transactions of the American Mathematical Society}

\begin{abstract}
The Torsion Anomalous Conjecture states that an irreducible variety $V$ embedded in a semi-abelian variety contains only finitely many maximal $V$-torsion anomalous varieties. In this paper we consider an irreducible variety embedded in a produc of elliptic curves. Our main result provides a totally explicit  bound for the N\'eron-Tate height of all maximal $V$-torsion anomalous points of relative codimension one, in the  non CM case, and an analogous effective result in the CM case. 
As an application, we obtain the  finiteness of such points. In addition, we deduce some new explicit results in the context of the effective Mordell-Lang Conjecture; in particular we bound the N\'eron-Tate height of the rational points of an explicit family of curves of increasing genus.
\end{abstract}
\maketitle

\section{Introduction}

 In this article, by \emph{variety} we mean an algebraic variety defined over the algebraic numbers. Equivalently, a variety $X$ is  defined by polynomials with coefficients in a number field $k$, however $k$ will not play any  role in our theorems. In addition, we identify $X=X(\overline{\qe})$.

Let $G$ be a semi-abelian variety.

A subvariety $V\subseteq G$  is a \emph{translate},  respectively  a \emph{torsion variety},  if it is a finite  union of translates of  proper algebraic subgroups  of $G$ by  points, respectively  by torsion points.

An irreducible variety $V\subseteq G$  is \emph{transverse}, respectively \emph{weak-transverse}, if  it is not contained in any translate, respectively in any torsion variety.

\medskip

Many important classical  results such as, for instance, the Manin-Mumford, the Mordell-Lang, the Bogomolov Conjectures, nowadays theorems, and  many open  problems, such as the Zilber-Pink Conjecture, investigate the relationship between these geometrical definitions and the arithmetical properties of the variety $V$.

Recently E. Bombieri, D. Masser and U. Zannier in \cite{BMZ1}  introduced the  notions of anomalous and torsion anomalous varieties and formulated some general conjectures. 
 Following their work, we  give the definition of torsion anomalous varieties. However, unlike \cite{BMZ1}, we allow torsion anomalous varieties to be zero dimensional.  Like this, we can simplify the formulation of several statements.

\medskip
 Let $V$ be a subvariety of a semi-abelian variety $G$. We say that
 an irreducible subvariety $Y$ of $V$ is a \emph{$V$-torsion anomalous} variety if 
\begin{enumerate}
\item[(i)] $Y$ is an irreducible component of $V\cap (B+\zeta)$ with $B+\zeta$ an irreducible  torsion variety of $G$;
\item[(ii)]  the dimension of $Y$ is larger than expected, i.e.  the codimensions satisfy  $${\mathrm {codim}\,} Y < {\mathrm {codim}\,} V +  {\mathrm {codim}\,}  B.$$
\end{enumerate}

We say that $B+\zeta$ is \emph{minimal} for $Y$ if, in addition, it has minimal dimension. The codimension of $Y$ in its minimal $B+\zeta$ is called the \emph{relative codimension} of $Y$.

We also say that a $V$-torsion anomalous variety $Y$ is \emph{maximal}  if it is not contained in any $V$-torsion anomalous variety of strictly larger dimension. 

\medskip

\medskip

The Torsion Anomalous Conjecture (TAC) is  a natural variant of a conjecture by Bombieri, Masser and Zannier.

\begin{con}[TAC] An  irreducible  subvariety  $V$  of a  semi-abelian variety  contains only finitely many maximal $V$-torsion anomalous varieties.\end{con}

Clearly, if $V$ is not weak-transverse, then $V$ is itself $V$-torsion anomalous. In addition, if $V$ is a hypersurface, the  only maximal $V$-torsion anomalous varieties are the maximal torsion varieties  contained in $V$,  and they are finitely many by the Manin-Mumford Conjecture. So the TAC is interesting when $V$ is weak-transverse of codimension at least $2$.

\medskip
The Zilber-Pink Conjecture  is a special case of the TAC. More  precisely, it is equivalent  to the TAC restricted only to  the $V$-torsion anomalous varieties that  come from an intersection expected to be empty.
So the TAC  implies,  like the Zilber-Pink Conjecture, several  other celebrated questions such as the Manin-Mumford and the Mordell-Lang  Conjectures.    Recent works also highlight links to model theory and to algebraic dynamics, in the context of  the Morton Conjectures.

\medskip

Only the following few cases of the TAC are known:  for curves  in a product of elliptic curves (Viada  \cite{ioant}), in abelian varieties with CM (R\'emond \cite{remond09})  and in a torus  (Maurin  \cite{Maurin}); for varieties of codimension  $2$ in a torus (Bombieri, Masser and Zannier  \cite{BMZ1}) and in a product of elliptic curves with CM  (\cite{TAI}).  Under  stronger geometric hypotheses on $V$, related results are proved by  many other authors.

\medskip

It is proven in several works that, if the height of a set of maximal torsion anomalous points is bounded, then such a set is also finite. So when dealing with points, the obstruction to the TAC is due only to  the lack of bounds for their height.  This leads us to state the following  natural extension of the  Bounded Height Conjecture (BHC), formulated by Bombieri, Masser and Zannier  \cite{BMZ1}.

\begin{con}[BHC']\label{BHC} For an irreducible  variety  $V$  in a  semi-abelian variety, the set of  maximal $V$-torsion anomalous points has bounded height.\end{con}
Like above, if $V$ is not weak-transverse then there are no maximal $V$-torsion anomalous points, because they are all contained in $V$, which is itself $V$-torsion anomalous. 
 
Some relevant  results in this context are due to  Habbeger (see \cite{habIMRN} and \cite{habInv}).  His theorems imply the BHC' in tori and abelian varieties, but only for the $V$-torsion anomalous points not contained in any $V$-anomalous variety.  This  condition is stronger than being maximal. 
In particular for varieties that do not satisfy a geometric condition even stronger than transversality his theorems are saying nothing (in his notations  the set $V^{oa}$ is the empty set).

 In addition, his results are not effective.
 Effective  theorems are  essentially only known for transverse curves in tori and in products of elliptic curves. 

In \cite{TAI}  the authors prove the BHC', but only for $V$-torsion anomalous points of relative codimension one when $V$ is a subvariety of a power of an elliptic curve with CM.
The method is  effective, however the use of a Lehmer type bound  is a deep obstacle  when trying to make the result explicit and such a bound  is not known in the non CM case.

\medskip

In this article we are  concerned with subvarieties of a power of an elliptic curve, regardless  of whether it has CM. From now on the ambient variety $G$ will be $E^N$, where $E$ is an elliptic curve defined over the algebraic numbers,  embedded in $\P_2$ via its Weierstrass equation, and $N$  is a positive integer. We consider on $E^N$ the canonical N\'eron-Tate height, denoted $\hat{h}$.

The  aim of this article is {twofold}: \begin{itemize}
\item  We prove the  BHC'  and  the TAC, for  maximal $V$-torsion anomalous points of relative codimension one. Our method is completely effective.
\item In the non CM case,  we make our proof explicit, computing all constants and obtaining the  only known bounds. 
\end{itemize}
The importance of giving an explicit result is due, for instance, to the implications on the Effective Mordell-Lang Conjecture specified below. More generally,  it is well known that an effective  TAC implies the  Effective Mordell-Lang Conjecture, which has a strong impact in mathematics and it is only known for curves in a torus. \\

Our main result is:

\begin{thm}\label{teorema}
 Let $V$ be an irreducible variety embedded in $E^N$.
Then the set of maximal  $V$-torsion anomalous points of relative codimension one has effectively bounded N\'eron-Tate height. If $E$ {is non} CM the bound is explicit, 
we have
\[
\hat h(P)\leq C_1(N)h(V)(\deg V)^{N-1}+C_2(E,N)(\deg V)^{N}+C_3(E,N),
\]
where 
\begin{align*}
C_1(N)&=(N!)^{N} N^{3N-2} \left( \frac{3^{N^2+N+1}2^{2N^2+3N-1}(N+1)^{N+1}}{(\omega_N \omega_{N-1})^2} \right)^{N-1}\\
C_2(E,N)&=C_1(N)\left(\frac{3^N \log 2}{2}+12 N \log 2+N\log 3+6N h_{\mathcal W}(E)\right)\\
C_3(E,N)&=\frac{7N^2}{6}\log 2+\frac{N^2}{2}h_{\mathcal W}(E),
\end{align*}
 $ \omega_r=\pi^{r/2}/\Gamma(r/2+1)$ is the volume of the euclidean unit ball in $\R^r$, $h(V)$ is the  normalised  height  of $V$ and $h_{\mathcal W}(E)$ is the height of the Weierstrass equation of $E$ (see Section \ref{SezioneAltezze} for  the definitions). 
\end{thm}

The bound of Theorem \ref{teorema} does not depend on the field of definition of $V$, unlike other bounds in similar  contexts. This is  central for our applications.

The structure of the  proof of Theorem \ref{teorema} does not distinguish CM from non CM elliptic curves. The constants can be computed also in the CM case, however further technical complications due to the structure of the endomorphism ring of $E$ would make the presentation less clear. For simplicity, we {prefer} to give the explicit computation only in the non CM case.

In \cite{IMRNViada} Theorem 1.1 (and a remark at page 1220 for the CM case), E. Viada  proved,  that on a weak-transverse variety $V\subseteq E^N$, the maximal $V$-torsion anomalous points of bounded height are finitely many.
 Since there are no maximal $V$-torsion anomalous  points if $V$ is not weak-transverse, we immediately deduce the following special case of the TAC.
\begin{cor}\label{corolIMRN}
An irreducible subvariety $V$ of $E^N$ contains only finitely  many maximal $V$-torsion anomalous points of relative codimension one.
\end{cor}

\medskip

As an application of our main theorem we obtain new explicit results in the context of the Effective Mordell-Lang Conjecture. 
In the setting of abelian varieties, the only known effective methods  for the Mordell-Lang Conjecture  are the Chabauty-Coleman method (see \cite{MP}) and the Manin-Demjanenko method (see \cite{serreMWThm}, Chapter 5.2). Both methods are quite difficult to apply; for some of the few explicit applications, see \cite{MP} and \cite{KuleszApplicazioneMD}.

In Section \ref{expli} we prove the following explicit theorem on points of rank 1;  the rank of a  point in $E^N$ is the rank of the subgroup of $E$ generated by its coordinates.

 \begin{thm}\label{corMW}
 Let $N\ge 3$ and let $\mathcal{C}\subseteq E^N$ be a weak-transverse curve.
 The set of points   $P\in \Ci$ of rank $\le 1$ 
 is a set of N\'eron-Tate height effectively bounded. If $E$ is non CM, we have that \begin{align*}
 \hat h(P)\leq C_1(N)h(\Ci)(\deg \Ci)^{N-1}+ C_2(E,N) (\deg \Ci)^{N}+C_3(E,N),
 \end{align*}
 where $C_1(N), C_2(E,N), C_3(E,N)$ are the same as in Theorem \ref{teorema}.

 If $\Ci$ is a transverse curve in $E^2$, then the set of points $P\in \Ci$ of rank $\le1$ is a set of N\'eron-Tate height effectively bounded. If $E$ is non CM, we have that
 \begin{align*}
 \hat h(P)\leq D_1h(\Ci)(\deg \Ci)^{2}+ D_2(E) (\deg \Ci)^{3}+D_3(E),
 \end{align*}
where
\begin{align*}
  D_1&=\frac{2^{64}3^{40}}{\pi^{8}} &\approx& 2.364\cdot 10^{34}\\
    D_2(E)&=\frac{2^{62}3^{41}}{\pi^{8}}\left(71\log 2+4\log 3+30h_{\mathcal W}(E)\right) &\approx& \left(5.319 \cdot h_{\mathcal W}(E) +9.504\right)\cdot 10^{35}\\
  D_3(E)&=\frac{9}{2}h_{\mathcal W}(E)+\frac{21}{2}\log 2 &\approx& 4.5\cdot h_{\mathcal W}(E)+7.279.
\end{align*}
In particular, in both cases, if $k$ is a field of definition for $E$ and $E(k)$ has rank 1, then all points in $\mathcal{C}(k)$  have N\'eron-Tate height effectively bounded as above.
 \end{thm}

The assumption $N\ge3$ is necessary for weak-transverse curves. Indeed any weak-transverse translate of $E^2$ (for example $E\times p$ with $p$ not a torsion point)  contains infinitely many points of rank 1 and of unbounded height (in the example the points $([n]p,p)$ for all natural $n$).

\smallskip

Finally,  we  give  explicit bounds in a specific family of  curves.
 This example is particularly interesting as it gives, at least in principle, an algorithm to find all their rational points.
Let $E$ be the elliptic curve defined by the equation $y^2=x^3+x-1$; the group $E(\Q)$ has rank 1 with generator $g=(1,1)$  (see Section \ref{expliCn}). We write 
\begin{align*}
y_1^2=x_1^3+x_1-1\\
y_2^2=x_2^3+x_2-1
\end{align*}
for the equations of $E^2$ in $\P_2^2$, using affine coordinates $(x_1,y_1)\times (x_2,y_2)$.
We have the following theorem, proved in  Section \ref{expliCn}. 
\begin{thm}\label{teoremacurveCn}
Let $E$ be the elliptic curve defined above, and consider the family of curves $\{\Ci_n\}_n$ with $\Ci_n\subseteq E^2$ defined via the additional equation $x_1^n=y_2$. Then for every $n\geq 1$, if $P\in \Ci_n(\mathbb{Q)}$ we have 
\[\hat{h}(P)\leq 8.253\cdot 10^{38} (n+1)^3\]
Moreover, writing $P=([a]g,[b]g)$,  the following inequalities hold 
\begin{equation*}
|a|\leq 7.037\cdot 10^{19}(n+1)
\end{equation*}
{and
\begin{equation*}
|b|\leq \left(\frac{3na^2}{2}+14\log 2+10\right)^{\frac{1}{2}}.
\end{equation*}}
\end{thm}

\medskip
Our  explicit results cannot be obtained with the method used in \cite{TAI}, because the Lehmer type bound is not known in the non CM case and anyway there are no published proofs  that such a bound is effective. The available proof could possibly be made explicit, but to get reasonable bounds it would be necessary to  avoid the use of a complicated descent argument, using instead  a much simpler induction, as done for tori by Amoroso and Viada in \cite{amoviadacomm}. Nevertheless, even with such  improvements, the constants so obtained would be far from being optimal. Probably the dependence on the dimension $N$ could not be improved further than $N^{N^N}$, which is of one exponential more than our bound.

\medskip

We now briefly describe the proof-strategy of our main result.

The proof of  Theorem  \ref{teorema} relies on an approximation process. Let $P$ be a point as in Theorem \ref{teorema}; in particular $P$ is a component of $V\cap (B+\zeta)$ for some torsion  variety $B+\zeta$. We will replace $B+\zeta$ with an auxiliary translate of the form $H+P$ in such a way that $P$ is still a component of $V\cap (H+P)$, and the degree and height  of $H+P$ can be controlled in terms  of $\hat h(P)$.
Using the properties of the height functions, we can in turn control the height of $P$ in terms of the height and degree of $H+P$ itself, and combining carefully these inequalities leads to the desired result.

This construction has been introduced in tori by {Habegger} \cite{hab}, Lemma 5. When adopting this strategy for subvarieties of a power of an elliptic curve several complications arise in computing degrees and heights.
While the degree of a subtorus can be easily related to the associated matrix, to compute the degree of a subvariety of $E^N$, it is necessary to fix an embedding of the ambient variety in a projective space and to study how geometrical and arithmetical objects behave under this embedding. This is done in Section  \ref{section3}. Concerning heights,  to make the results explicit, we need to work with different heights functions  and use explicit versions of several bounds relating the height functions on $E^N$ and those in the projective spaces $\P_2^N$ and $\P_m$. 
Furthermore, we need to adapt and simplify some of the arguments in \cite{hab}, in order to keep the constants as small as possible.  This is done in  Section \ref{proofmain}. 

\medskip

The following is an outline of the content of the different sections of this paper.

In Section \ref{prelim-H}  we recall some classical results such as  the  Arithmetic B\'ezout Theorem,  the Zhang Inequality and  the Minkowski Theorem. We also present the geometrical setting  and we give explicit bounds relating different height functions.
Moreover, we recall the correspondence between  algebraic subgroups of $E^N$ and matrices with coefficients in $\emor(E)$.

In Section \ref{tai} we give the structure of the proof of Theorem \ref{teorema} while postponing to Sections \ref{section3} and \ref{proofmain} the proof of the technical step. 

 In Section \ref{expli}  we  give the proof of Theorem \ref{corMW} and in Section \ref{expliCn} we prove Theorem \ref{teoremacurveCn}.

In Section \ref{section3} we  compute explicit bounds for the degree of the rational functions that represent morphisms  from $E^N$ to $E$.

In Section \ref{proofmain}, for a torsion anomalous point $P$ of relative codimension one, we give the construction of the auxiliary translate $H+P$, showing how to bound its height and degree.

\section{Embeddings, heights and algebraic subgroups}\label{prelim-H}
Let $E$ be an elliptic curve without complex multiplication.  We fix a Weierstrass equation
\begin{equation*}
E: y^2=x^3+Ax+B
\end{equation*}
with $A$ and $B$  algebraic integers (this hypothesis is not restrictive).  As usual, we define 
\[
 \Delta=-16(4A^3+27B^2) \qquad j=\frac{-1728(4A)^3}{\Delta}.
\]

{In this section we give explicit bounds for embeddings of varieties in the projective space. In Subsection \ref{sezioneSegre} we compute the  degree of $E^N$ as a subvariety of $\P_{3^N-1}$, via the Segre embedding.}
 In Subsection \ref{SezioneAltezze}, we define several height functions, state their relevant properties and give explicit bounds between different heights. We also recall the Arithmetic B\'ezout Theorem and  the Zhang Inequality.  In Subsection \ref{prelim-S} we recall the relations between algebraic subgroups and matrices, degrees and  minors.

\subsection{Segre embedding}\label{sezioneSegre}
Let us consider the composition of maps
\begin{equation*}
  E^N \hookrightarrow \P_2^N \hookrightarrow \P_{3^N-1}.  
\end{equation*}
{The first map sends a point $(X_1,\ldots,X_N)$ to $((x_1,y_1),\ldots(x_N,y_N))$ where $(x_i,y_i)$ are the affine coordinates of $X_i$ in the Weierstrass form of $E$.} The second {map} is the Segre embedding. When computing heights and degrees of points and subvarieties, we will think them as embedded in $\P_{3^N-1}$ via the previous map.

The degree of a variety $V\subseteq \P_m$, in particular, is the maximal cardinality of a finite intersection $V\cap L$, with $L$ a linear subspace of dimension equal to $\cod V$,  the codimension of $V$. This degree is often conveniently computed as an intersection product.

\medskip

Let $X(E,N)$ be the variety in $\mathbb{P}_{3^N-1}$ identified with $E^N$ via the Weierstrass form and the Segre embedding.
The following lemma {computes} its degree. 
\begin{lem}\label{defic2} Let us denote by $\cuno$ the degree of $X(E,N)$. Then
\begin{equation}\label{valc4}\cuno=3^N N!.\end{equation}
\end{lem}
\begin{proof}
We can compute this degree by means of the intersection product in $\P_2^N$.
By definition, the degree of $X(E,N)$ is obtained intersecting it with $N$ hyperplanes in general position in $\mathbb{P}_{3^N-1}$, and computing the degree of the cycle thus obtained in the Chow ring. {Let $l$ be the class of a line in the Picard group of $\P_2$, and $l_i$ its pullback $\pi_i^*(l)$ through the projection $\pi_i:\P_2^N\to\P_2$ on the $i$-th factor. Then the Chow ring of $\P_2^N$ is described as $\Z[l_1,\dotsc,l_N]/(l_1^3,\dotsc,l_N^3)$. The class of  a hyperplane of $\mathbb{P}_{3^N-1}$ restricts to the element $l_1+\dotsb+l_N$, and the class of $E^N$ is easily seen to be $(3l_1)\dotsm (3l_N)$.} The desired intersection is therefore
\[(3l_1)\dotsm (3l_N)(l_1+\dotsb+l_N)^N=3^N N! (l_1\dotsm l_N)^2,\]
and the statement follows.
\end{proof}

\subsection{Heights}\label{SezioneAltezze}
 We need to work with different height functions. These height functions are all related to one another by effective relations. Making these relations explicit for applications is sometimes a delicate task. In this section, based on the work of Silverman and Zimmer, we are going to make explicit the constants that we will need.

\medskip 

Let $\Mm_K$ be the set of  places of a number field $K$.
For a point $P=(P_0:\dotsb :P_m)\in \P_m(K)$ let
\begin{equation}\label{defiH}
h(P)=\sum_{v\in\Mm_K}\frac{[K_v:\Q_v]}{[K:\Q]}\log \max_i \{\abs{P_i}_v\}
\end{equation}
be the logarithmic Weil height, and let
\begin{equation}\label{defiH2}
h_2(P)=\sum_{v\text{ finite}}\frac{[K_v:\Q_v]}{[K:\Q]}\log \max_i \{\abs{P_i}_v\} +\sum_{v\text{ infinite}}\frac{[K_v:\Q_v]}{[K:\Q]}\log \left(\sum_i \abs{P_i}_v^2\right)^{1/2} 
\end{equation}
be a modified version of the height that differs from the Weil height at the archimedean places.
They are both well-defined,  they extend to $\overline\Q$ and it follows easily from their definitions that
\begin{equation}\label{stima_altez}
h(P)\leq h_2(P)\leq h(P)+\frac{1}{2}\log(m+1).\end{equation}

\medskip

For an algebraic {number $x\in K$,}  the logarithmic Weil height (for short Weil height) $h(x)$ {is} the logarithmic Weil height  of the projective point $(x:1)\in\P_1$. We also denote $h_{\infty}(x)$ the contribution to the Weil height coming from the archimedean places, namely:
\[
 h_\infty(x)=\sum_{v\text{ infinite}}\frac{[K_v:\Q_v]}{[K:\Q]} \max \{\log \abs{x}, 0\}.
\]

\medskip

For a point $P$ on $E$, {$\hat h(P)$ is the} canonical N\'eron-Tate height, which is related to the Weil height of the $x$ coordinate  {of $P$} by the following bound (\cite{SilvermanDifferenceHeights}, Theorem 1.1):

\begin{equation}\label{BoundSilvermanAltezze}
 -\frac{h(j)}{24}-\frac{h(\Delta)}{12}-\frac{h_\infty(j)}{12}-0.973\leq \hat h(P)-\frac{1}{2}h(x(P))\leq \frac{h(\Delta)}{12}+\frac{h_\infty(j)}{12}+1.07.
\end{equation}
{To relate $\hat h(P)$ and the Weil height of the point $P$,  we define the \emph{Weil height of the Weierstrass equation of $E$} 
\begin{equation*}
h_{\mathcal W}(E)=h(1:A^{1/2}:B^{1/3})
\end{equation*}
as the Weil height of the projective point $(1:A^{1/2}:B^{1/3})\in\P_2$. Then by
\cite{ZimmerAltezze},  p.~40, we have}
\begin{equation}\label{stima_zimmer}
-\frac{h_{\mathcal W}(E)}{2}-\frac{7}{6}\log 2\leq \frac{1}{3} h(P)-\hat h(P)\leq h_{\mathcal W}(E)+2\log 2.
\end{equation}

\medskip

For a point $P=(P_1,\dotsc,P_r)\in\P_{d_1-1}\times\dotsb\times\P_{d_r-1}$ we define $h(P)$ and $h_2(P)$ applying formulae \eqref{defiH} and \eqref{defiH2} to the image of $P$ in $\P_{d_1 \dotsm d_r -1}$ via the Segre embedding. With these definitions the following relation holds:
\[h(P)=\sum_{i=1}^r h(P_i).\]

{For a point} $P=(P_1,\dotsc,P_r)\in E^r$, {the canonical height $\hat h(P)$ is}  the sum
\begin{equation}\label{sumhcan}
 \hat h(P)=\sum_{i=1}^r \hat h(P_i).
\end{equation}
In particular, combining \eqref{stima_zimmer} with \eqref{stima_altez}  we get that for every point $P\in E^N$ we have
\begin{equation}\label{stima_alte}\hat{h}(P)\leq \frac{h_2(P)}{3}+\cdue(E,N),\end{equation}
where \begin{equation}\label{valC3}\cdue(E,N)=\frac{N}{2}h_{\mathcal W}(E)+\frac{7N}{6}\log 2.\end{equation}

 From \eqref{stima_altez} and \eqref{stima_zimmer}, for every point $P$ of $E^N$  we also have
\begin{equation}\label{altezza2altezza^}
h_2(P)\leq h(P)+\frac{N}{2}\log 3\leq 3 \hat h(P) +\cquindici(E,N),
\end{equation}
where 
\begin{equation}\label{valc3}\cquindici(E,N)=N(3h_{\mathcal W}(E)+6\log 2+\frac{1}{2}\log 3).\end{equation}

\medskip

For a subvariety $V\subseteq\P_m$  we {consider} the  normalised height of $V$, denoted $h(V)$, {defined}  in terms of the Chow form of the ideal of $V$, as done  in \cite{patriceI} and \cite{patrice}. We remark that, with this definition, the height of a point $P$ regarded as a 0-dimensional variety is equal to the  height $h_2(P)$ previously defined, which is not equal, in general, to the Weil height of the point. To avoid confusion, we will always write $h_2(P)$ to denote the height of the variety $\{P\}$.

\bigskip

We end this subsection by recalling two classical results on the normalised height, the Arithmetic B\'ezout Theorem and the Zhang Inequality.
\begin{thm}[Arithmetic B\'ezout Theorem]\label{AriBez}
 Let $X$ and $Y$ be irreducible closed subvarieties of $\P_m$ defined over $\overline\Q$. If $Z_1,\dotsc,Z_g$ are the irreducible components of $X\cap Y$, then
 \[
  \sum_{i=1}^g h(Z_i)\leq\deg(X)h(Y)+\deg(Y)h(X)+\frac{(m+1)\log 2}{2}\deg(X)\deg(Y).
 \]
\end{thm}
For the constant $\frac{(m+1)\log 2}{2}$ {see \cite{BGSGreen}, Theorem 5.5.1 (iii).}
\begin{thm}[Zhang's inequality]\label{ZhangIneq}
Let $X\subseteq\P_m$ be an {irreducible algebraic subvariety.} Defining the essential minimum of $X$ as 
\[
 \mu(X)=\inf\{\theta\in\R\mid\{P\in X\mid h_2(P)\leq\theta\}\text{ is Zariski dense in }X\},
\]
 we have
\begin{equation*}
 \mu(X)\leq\frac{h(X)}{\deg X}\leq(1+\dim X)\mu(X).
\end{equation*}
\end{thm}
We also define a different essential minimum for subvarieties of $E^N$, that will be used in Subsection \ref{subsec-alttrasl}, as
\[
 \hat \mu(X)=\inf\{\theta\in\R\mid\{P\in X\mid \hat h(P)\leq\theta\}\text{ is Zariski dense in }X\}.
\]

By \eqref{stima_alte} and \eqref{altezza2altezza^}, these two definitions are related by the inequality
\begin{equation}\label{mu2mu^}
3\hat\mu(X)-3\cdue(E,N)\leq \mu(X)\leq 3 \hat\mu(X) + \cquindici(E,N)
\end{equation}
 where the constants $\cdue(E,N)$ and $\cquindici(E,N)$ are defined in \eqref{valC3} and \eqref{valc3} respectively.

\subsection{Algebraic Subgroups}\label{prelim-S}
{In this subsection we present the relationship between several different descriptions of the  algebraic subgroups of $E^N$.

Before explaining it in detail, we need to recall some classical tools in the geometry of numbers. Let $r$ and $N$ be positive integers, with $r\leq N$, and let $\Lambda$ be a lattice of rank $r$ in $\mathbb{R}^N$.  We define the determinant of $\Lambda$ as $\det\Lambda=\sqrt{\det(MM^t)}$, where $M\in\mathrm{Mat}_{r\times N}(\mathbb{R})$ is any matrix whose rows form a basis of $\Lambda$. We also define the successive minima $\lambda_i$ of $\Lambda$ as
\[
 \lambda_i=\inf\{t\in\R\mid\dim \langle B_t\cap\Lambda\rangle_\R=i\},
\]
where $B_t$ is the euclidean ball or radius $t$ centered at the origin.

The following theorem by Minkowski plays an important role in this setting.

\begin{thm}[Minkowski's second theorem]
Let $\Lambda$ be a lattice  of rank $r$. Then 
\begin{equation*}
 \frac{2^r}{r!}\det \Lambda \leq \omega_r \lambda_1\dotsm\lambda_r \leq 2^r \det \Lambda,
\end{equation*}
where the $\lambda_i$'s are the successive minima of $\Lambda$ and 
\begin{equation}\label{valomega}
 \omega_r=\frac{\pi^{r/2}}{\Gamma(r/2+1)}
\end{equation}
is the volume of the euclidean unit ball in $\R^r$ (here $\Gamma$ denotes the Euler $\Gamma$ function).
\end{thm}

\bigskip

It is well known that abelian subgroups of codimension $r$, matrices of rank $r$ in $\mathrm{Mat}_{r\times N}(\mathrm{End}(E))$, and lattices of rank $r$ in $(\mathrm{End}(E))^N$ are essentially representations of the same objects up to some torsion subgroup; analogously, their degree, minors and  successive minima can be related up to constants. 

In more  detail,
let $B+\zeta$ be an irreducible torsion variety of $E^N$  of codimension $\codim\sotto=r$ and let $\pi_B:E^N\to E^N/B$ be the natural projection.
We know that $E^N/B$ is isogenous to $E^r$; let $\varphi_\sotto:E^N \to E^{r}$ be the composition of $\pi_B$ and this isogeny.

We  associate  $\sotto$ with the morphism $\varphi_\sotto$ and we have that $\ker \varphi_\sotto=\sotto+\tau$ with $\tau$ a torsion  subgroup whose cardinality is absolutely bounded (by \cite{Masserwustholz} Lemma 1.3).}
Obviously $\varphi_\sotto$ is identified with a matrix in $\mathrm{Mat}_{r\times N}(\mathrm{End}(E))$ of rank $r$.   Using basic geometry of numbers, we can choose the matrix   representing $ \varphi_\sotto$ such that the degree of   $\sotto$ is essentially the product of  the squares of the norms  of the rows of the matrix.

More precisely, given $B\subseteq E^N$ an algebraic subgroup of rank $r$, we associate it with a matrix in $\mathrm{Mat}_{r\times N}(\mathrm{End}(E))$ with rows  $u_1,\dotsc,u_r$ 
such that the euclidean norm $|u_i|$ of $u_i$ equals the $i$-th successive minimum of the lattice $\Lambda=\langle u_1,\dotsc, u_r\rangle_{\mathbb{Z}}$. In Subsection \ref{sec-c6} we show that there is a constant $\ccinque(N,r)$ such that
\begin{equation}\label{gradorighe}
 \deg B\leq \ccinque(N,r)\prod_{i=1}^r|u_i|^2
\end{equation}
and, when $E$ {is non} CM, we have
\[\ccinque(N,r)=3^N N!\left(\frac{3}{2}(N+1)12^{N-1}\right)^r.\]

Combining bound \eqref{gradorighe} and Minkowski's theorem one can relate the degree of an  algebraic subgroup and the determinant of the associated lattice (\emph{i.e.} the lattice generated by the rows of the associated matrix);  when the curve $E$ is non CM the following explicit bound holds:
\begin{equation}\label{gradodet}
 \deg B\leq \ccinque(N,r)\frac{4^r}{\omega_r^2}(\det\Lambda)^2,
\end{equation}
 where $\ccinque(N,r)$ is given above.

\section{The main result}\label{tai}
In this section we give the proof of our main result, which is Theorem \ref{teorema} of the Introduction. 
The constants  that we obtain are always effective and also explicit in the non CM case.

{

The proof of our main theorem is  based on the following idea: given a point $P\in E^N$ which is contained in a torsion variety of dimension $1$, we construct, by means of the geometry of numbers, another abelian subvariety $H\subseteq E^N$ of dimension $1$ so that the degree $\deg H$ and the height of the translate $H+P$ are both well controlled. An application of the Arithmetic B\'ezout Theorem, recalled in Subsection \ref{SezioneAltezze}, leads then to the end of the proof.
We will show in the technical Section \ref{proofmain} how to construct the auxiliary algebraic subgroup $H$; in order to compute all the constants explicitly we need several pages of careful  computations, which we have collected in Section \ref{section3}.

The overall construction of the auxiliary  algebraic subgroup $H$ is summarised in the following propositions, whose proofs are postponed to  Section \ref{proofmain}.
}

\begin{propo}[Non CM Case]\label{main}
Let $E$ be a {non CM} elliptic curve and let $1\leq m \leq N$ be  integers. Let $P=(P_1,\dotsc,P_N)\in B\subseteq E^N$, where $B$ is a torsion variety of dimension $\leq m$.
Say $s$ is an integer with $1\leq s\leq N$ and $T\geq 1$ a real number.

Then there exists an abelian subvariety $H$ of codimension $s$ such that
\begin{align*}
\deg(H+P)&\leq \ccinque(N,s) T\\
h(H+P)&\leq \ctre(N,m,s)T^{1-\frac{N}{ms}}\hat h(P)+\cquattro(E,N,s)T;
\end{align*}
 where 
 \begin{align*}
  \ccinque(N,s)&=3^N N! \left( \frac{3}{2}(N+1)12^{N-1}\right)^s\\
  \ctre(N,m,s)&=m^3 (m!)^4 \binom{N+m}{N}\frac{3sN^2(N-s+1)4^{3N-m+1}}{(\omega_s\omega_{N-s}\omega_N)^2}\ccinque(N,s)\\
  \cquattro(E,N,s)&=3N(N-s+1)\left(2\log 2+\frac{\log 3}{6}+ h_{\mathcal W}(E)\right)\ccinque(N,s).
 \end{align*}
  Here $ \omega_r=\pi^{r/2}/\Gamma(r/2+1)$ is the volume of the euclidean unit ball in $\R^r$, and $h_{\mathcal W}(E)$ is defined in Section \ref{SezioneAltezze}.
\end{propo}

\begin{propo}[CM Case]\label{mainCM}
Let $E$ be a CM elliptic curve  and let $1\leq m \leq N$ be  integers. Let $P=(P_1,\dotsc,P_N)\in B\subseteq E^N$, where $B$ is a torsion variety of dimension $\leq m$.
Say $s$ is an integer with $1\leq s\leq N$ and $T\geq 1$ a real number.

Then there exists an abelian subvariety $H$ of codimension $s$ such that
\begin{equation*}
h(H+P)\leq \csei T^{1-\frac{N}{ms}}\hat h(P)+\csette T
\end{equation*}
and
\begin{equation*}
\deg(H+p)\leq \cotto T,
\end{equation*}
where $\csei, \csette,\cotto$ are effective positive constants depending only  on  the integers $N,m,s$, the ring $\mathrm{End}(E)$ and the height $h_{\mathcal W}(E)$ (defined in Section \ref{SezioneAltezze}).
\end{propo}

We now show how to deduce {Theorem} \ref{teorema} from Propositions \ref{main} and  \ref{mainCM}.

\begin{proof}[Proof of Theorem \ref{teorema}]
    If $N=2$, then the codimensional inequality (ii) at page 1  tells us that  the only $V$-torsion anomalous points  are the torsion points contained in $V$, which have height zero. We can now assume that $N\geq 3$.

The point $P$ is a component of the intersection $V\cap (B+\zeta)$, where $B+\zeta$ is a torsion variety of $\dim B=1$.

Assume first that $E$ is non CM.
Let $T$ be a free parameter that will be specified later; we apply Proposition \ref{main} to $P$, $T$, $m=1$ and $s=N-1$. This gives a translate $H+P$ of dimension $\dim (H+P)=1$, of degree bounded in terms of $T$ and  such that $h(H+P)$ is bounded solely in terms of $\deg H$ and $\hat h(P)$.

Explicitely, if 
\begin{align*}
 \cdieci(N)&=\ccinque(N,N-1)=3^N N! \left(\frac{3}{2}(N+1)12^{N-1}\right)^{N-1}\\
 \cnove(N)&=\ctre(N,1,N-1)=\frac{3}{2}\frac{N^2(N^2-1)64^N}{(\omega_N\omega_{N-1})^2}\cdieci(N)\\
 \cundici(E,N)&=\cquattro(E,N,N-1)=6N(h_{\mathcal W}(E)+2\log 2+\frac{1}{6}\log 3)\cdieci(N),
\end{align*}
then the degree and the height of the translate $H+P$  are bounded by  Proposition \ref{main} as
 \begin{equation*}
 \deg (H+P)\leq \cdieci(N)T
\end{equation*}
 and 
\begin{equation*}
 h(H+P)\leq \frac{\cnove(N)}{ T^{1/(N-1)}}\hat h(P)+\cundici(E,N) T.
\end{equation*}

We want to prove that $P$ is a component of $V\cap (H+P)$.
If not, then $H+P\subseteq V$  because $\dim (H+P)=1$. In addition $\dim(B+H+\zeta)\leq 2$, as $\dim B=1$. Since $N\ge3$, the torsion variety $B+H+\zeta$ is proper. Thus $H+P\subseteq V\cap (B+H+\zeta)$ would be $V$-torsion anomalous, contradicting the maximality of $P$.

This means that $P$ is a component of $V\cap (H+P)$. In order to bound the height of $P$ we can apply the  Arithmetic B\'ezout Theorem to the irreducible varieties $V$ and $H+P$. We have 
\begin{equation}\label{boundpunto51b}
h_2(P)\leq h(V)\cdieci(N)T+\deg V \left( \frac{\cnove(N)}{ T^{1/(N-1)}}\hat h(P)+\cundici(E,N) T \right) +\frac{3^N \log 2}{2} \cdieci(N) T \deg V.
\end{equation}

We now choose 
\begin{equation*}
 T=\left(\frac{N}{N-1}\frac{\cnove(N)}{3} \deg V\right)^{N-1},
\end{equation*}
so that the coefficient of $\hat h(P)$ at the right-hand side of \eqref{boundpunto51b} becomes $3(N-1)/N$. Recall that by \eqref{stima_alte}
\[\hat{h}(P)\leq \frac{h_2(P)}{3}+\cdue(E,N)\]
where $\cdue(E,N)$ is the explicit constant in \eqref{valC3} depending only on the coefficients of $E$ and on $N$. Then we get
\begin{align*}
3 \hat{h}(P)\leq& \frac{3(N-1)}{N} \hat{h}(P)+3\cdue(E,N)+\cdieci(N)Th(V) +\\
&+\left(\frac{3^N \log 2}{2} \cdieci(N) +\cundici(E,N)\right)T \deg V,
\end{align*}
and hence 
\begin{equation*}
 \hat{h}(P)\leq \frac{N}{3}\cdieci(N)Th(V)+\frac{N}{3}\left(\frac{3^N \log 2}{2} \cdieci(N) +\cundici(E,N)\right)T \deg V+N\cdue(E,N),
\end{equation*}
{which is the desired bound} for $\hat{h}(P)$.

This concludes the proof of the non CM case.
Notice that the application of Proposition \ref{main} is the only point in the proof where it is required that $E$ {is non} CM. Therefore the same argument, with the use of Proposition \ref{mainCM} instead of Proposition \ref{main}, proves the result in the CM case.
\end{proof}

\section{An application to the Effective Mordell-Lang Conjecture }\label{expli}

We now clarify the implications of our theorems on the Effective Mordell-Lang Conjecture, proving Theorem \ref{corMW} from the introduction.

 \begin{proof}[Proof of Theorem \ref{corMW}]
   
 The last part of the theorem is a direct consequence of the first part. Indeed if $E(k)$ has rank 1, then all points in $\Ci(k)$ have rank at most one.

We now prove the  first part of the theorem.
The points of rank zero are exactly the  torsion points; for these points the bound is trivially true because their height is zero. 

Consider first the case that $\Ci$ is weak-transverse and $N\ge3$. 
Let $P=(P_1,\dotsc,P_N)\in\Ci$ be a point of rank 1 and let  $g\in E$ be a generator of  $\Gamma_P=\langle P_1, \dots, P_N\rangle_\qe$. Then the coordinates of $P$ satisfy $a_i P_i=b_i g$ for some $a_i\neq 0$ and $b_i$ in $\mathbb{Z}$.
 Since  $P$ is not a torsion point, at least one of the $b_i$ must be different from zero; let's say that $b_1\neq 0$.
 Then $P$ lies on the algebraic subgroup $B$  in $E^{N}$ given by the intersection of the $N-1$ algebraic subgroups of equations $a_i b_1 X_i=a_1 b_i X_{1}$ for $i=2,\dotsc,N$. Note that the matrix of coefficients has obviously rank $N-1$, so the dimension of $B$ is one.
 
 Since $\Ci$ is weak-transverse, $P$ is a component of $\Ci\cap B$; thus for $N\ge3$ $P$ is $\Ci$-torsion anomalous and it has relative codimension $1$.
 In addition, on  weak-transverse curves all torsion anomalous points are maximal; thus $P$ is a maximal $\Ci$-torsion anomalous point of relative codimension 1.
We can now apply Theorem \ref{teorema} to $V=\Ci$ to obtain the height bound,   thus concluding the case of $N\geq 3$.

For $N=2$, the previous argument cannot be directly applied, indeed  a point of rank 1 is never torsion anomalous in $E^2$. We now show how to reduce the case of a transverse curve $\Ci\subseteq E^2$ to the previous case.
Let $P=(P_1,P_2)\in\Ci$ be a point of rank 1 and let  $g\in E$ be a generator of  $\Gamma_P=\langle P_1, P_2\rangle_\qe$.
Fix a positive real $\epsilon$, and choose an integer $M$ such that $\epsilon M^2\geq \hat h(g)\deg(\Ci)$.
Let $Q_M$ be a point in $E$ such that $MQ_M=g$; by our choice of $M$ we have that 
\[\hat h(Q_M)=\frac{\hat h(g)}{M^2}\leq \frac{\epsilon}{\deg \Ci}.\]
We define $\Ci_M'=\Ci\times \{Q_M\}\subseteq E^{3}$ and $P_M'=P\times \{Q_M\}\in \Ci_M'$.
Since $\Ci$ is transverse in $E^2$, $\Ci_M'$ is  weak-transverse in $E^3$. Notice that  $P_M'\in \Ci_M'$ is a point of rank $1$ and  $$\hat h(P)\leq \hat h(P_M').$$

 In addition $\deg \Ci_M'=\deg \Ci$  and $\hat\mu(\Ci_M')=\hat\mu(\Ci)+\hat h(Q_M) $. By Zhang's inequality and \eqref{mu2mu^} we have
\begin{align*}\label{altezza}
h(\Ci_M')\leq 2\mu(\Ci_M')\deg\Ci\leq 2\deg\Ci\left(3\hat\mu(\Ci_M')+\cquindici(E,3)\right)=\\
=2\deg\Ci\left(3\hat\mu(\Ci)+3\hat h(Q_M)+\cquindici(E,3)\right)\leq\\
\leq 2\deg\Ci\left(\mu(\Ci)+3\cdue(E,3)+3\hat h(Q_M)+\cquindici(E,3)\right)\leq\\
\leq 2h(\Ci)+6\epsilon+6\deg\Ci\left(\cdue(E,3)+\frac{\cquindici(E,3)}{3} \right)
\end{align*}
where the constant $\cdue(E,3)$ is defined in \eqref{valC3} and $\cquindici(E,3)$ in \eqref{valc3}.

To bound $\hat{h}(P_M')$ and in turn $\hat{h}(P)$, we apply the first part of the theorem to  $\Ci_M' \in E^3$, obtaining:
\begin{align*} \hat{h}(P)\le  & 2C_1(3)h(\Ci)(\deg\Ci)^2+\left(C_2(E,3) +6\cdue(E,3)C_1(3)+2\cquindici(E,3)C_1(3) \right)(\deg\Ci)^3+\\
&+C_3(E,3)+6\epsilon  C_1(3)(\deg\Ci)^2.
\end{align*}

Clearly the point $P$  does not depend on the initial choice of $\epsilon$ and, letting $\epsilon$ go to zero, we get the  desired bound for the height.
 \end{proof}

\section{Rational points on an explicit family of curves}\label{expliCn}

  We now give an explicit method to find, in principle, all rational points on a family of curves in a power of a non CM elliptic curve.

Let $E$ be the elliptic curve defined by the Weierstrass equation
\[E: y^2=x^3+x-1.\]
With an easy computation one can check that 
\begin{align*}
\Delta(E)&=-496,\\
j(E)&=\frac{6912}{31},\\
h_\mathcal{W}(E)&=0,
\end{align*}
in particular the curve {is non} CM because $j(E)\not\in \Z$.
Furthermore the group $E(\Q)$ has rank 1 with generator $g=(1,1)$ and no non-trivial torsion points; this can be checked on a database of elliptic curve data  (such as \url{http://www.lmfdb.org/EllipticCurve/Q}).
The N\'eron-Tate height of the generator $g$ can be bounded from below, computationally, as
\begin{equation}\label{hg}
\hat h(g)\geq 1/4
\end{equation} (we used dedicated software (PARI/GP) which implements an algorithm with sigma and theta functions due to Silverman).

As  an application of our main result we give the proof of Theorem \ref{teoremacurveCn}, which is an example of the explicit Mordell Conjecture for a family of curves in $E^2$ of increasing genus and degrees. We recall the definition of the curves from the introduction. 
We write 
\begin{align*}
y_1^2=x_1^3+x_1-1\\
y_2^2=x_2^3+x_2-1
\end{align*}
for the equations of $E^2$ in $\P_2^2$, using affine coordinates $(x_1,y_1)\times (x_2,y_2)$ and we consider the family of curves $\{\Ci_n\}_n$ with $\Ci_n\subseteq E^2$ defined via the additional equation $$x_1^n=y_2.$$

\begin{proof}[Proof of Theorem \ref{teoremacurveCn}]
To prove the theorem we first show that the curves $\Ci_n$ are  irreducible and transverse in $E^2$ and then we apply Theorem  \ref{corMW} with $N=2$  and $E(\qe)$ of rank 1, computing all invariants of the case.

 The irreducibility of the $\Ci_n$ is easily seen to be equivalent to the primality of the ideal generated by the polynomials $y_1^2-x_1^3-x_1+1$ and $x_1^{2n}-x_2^3-x_2+1$ in the ring $\qe[x_1,x_2,y_1]$. This follows from an easy argument in commutative algebra. 

 Notice that in $E^2$ the only irreducible curves that are not transverse are translates, so curves of genus one. Thus, we need to show only that each $\Ci_n$ has genus at least 2;  in fact we  prove  that $\Ci_n$ has genus $4n+2$.

\medskip

Consider the morphism $\pi_n:\Ci_n\to\P_1$ given by the function $y_2$. The morphism $\pi_n$ has degree $6n$, because for a generic value of $y_2$ there are three possible values for $x_2$, $n$ values for $x_1$, and two values of $y_1$ for each $x_1$.

Let $\alpha_1,\alpha_2,\alpha_3$ be the three distinct roots of the polynomial $f(T)=T^3+T-1$; {let also $\beta_1,\beta_2,\beta_3,\beta_4$ be the four roots of the polynomial $27T^4+54T^2+31$, which are the values such that $f(T)-\beta_i^2$ has multiple roots.} Notice that none of the $\alpha_i$ can be equal to {the $\beta_j$} because they have different degrees.

The morphism $\pi_n$ {is ramified over $\beta_1,\beta_2,\beta_3,\beta_4,0,\alpha_1^n,\alpha_2^n,\alpha_3^n,\infty$.}
{Each of the points $\beta_i$ has} $2n$ preimages of index 2 and $2n$ unramified preimages. The point 0 has 6 preimages ramified of index $n$. The points $\alpha_i^n$ have 3 preimages ramified of index 2 and $6n-6$ unramified preimages.
The point at infinity is totally ramified.

By Hurwitz formula
\begin{align*}
2-2g(\Ci_n)&=\deg\pi_n (2-2g(\P_1))-\sum_{P\in \Ci_n}(e_P-1)\\
2-2g(\Ci_n)&=12n-(4\cdot 2n +6(n-1)+3\cdot 3 +6n -1)\\
g(\Ci_n)&=4n+2.
\end{align*}

Thus the family $\{\Ci_n\}_n$ is a family of transverse curves in $E^2$.
We can therefore apply Theorem \ref{corMW} with $N=2$ to each $\Ci_n$, which gives, for 
 $P\in \Ci_n(\mathbb{Q})$
 \begin{equation}\label{bhpunto}\hat h(P)\leq\left(2.364\cdot 10^{34}h(\Ci_n)+9.504\cdot 10^{35}\deg \Ci_n\right)(\deg \Ci_n)^2.\end{equation}

We now compute $\deg \Ci_n$ and $h(\Ci_n)$.

\medskip

 We can compute the degree of $\Ci_n$ as an intersection product. {Let $\ell, m$ be the classes of lines of the two factors of $\P_2^2$ in the Chow group. Then the degree of $\Ci_n$ is obtained multiplying the classes of the hypersurfaces cut by the equation $x_1^n=y_2$, which is $n\ell+m$, by the two Weierstrass equations of $E$, which are $3\ell$ and $3m$, and by the restriction of an hyperplane of $\P_8$, which is $\ell+m$.
In the Chow group
\[(n\ell+m)(3\ell)(3m)(\ell+m)=9(n+1)(\ell m)^2\] 
and then
\[\deg\Ci_n=9(n+1).\]
}
\smallskip

We estimate the  height of $\Ci_n$  using  Zhang's inequality and computing an upper bound for the essential minimum $\mu(\Ci_n)$ of $\Ci_n$.   To this aim, we construct an infinite set of points on $\Ci_n$ of bounded height. By the definition of essential minimum, this gives also an upper bound for $\mu(\Ci_n)$.

 Let $Q_\zeta=((x_1,y_1), (\zeta,y_2))\in \Ci_n$, where $\zeta \in \overline{\mathbb{Q}}$ is a root of unity. Clearly there exist infinitely many such points on $\Ci_n$. Denoting by $h$ the logarithmic Weil height on $\overline{\mathbb{Q}}$ and using the equations of $E$ and $\Ci_n$, we have:
\[h(\zeta)=0, h(y_2)\leq \frac{\log 3}{2}, h(x_1)\leq \frac{\log 3}{2n}, h(y_1)\leq \frac{\log 3}{n}+\frac{\log 3}{2}.\] 
 Thus
\[ h(x_1,y_1)\leq \log 3\left(\frac{n+3}{2n}\right), h(\zeta,y_2)\leq \frac{\log 3}{2}\]
and from \eqref{stima_altez}
\[ h_2(x_1,y_1)\leq \log 3\left(\frac{2n+3}{2n}\right), h_2(\zeta,y_2)\leq \log 3.\]
 So for all points $Q_\zeta$ we have
\[h_2(Q_\zeta)=h_2(x_1,y_1)+h_2(\zeta,y_2)\leq \log 3\left(\frac{4n+3}{2n}\right).\]  By the definition of essential minimum, we deduce \[\mu(\Ci_n)\leq \log 3\left(\frac{4n+3}{2n}\right)\]
and by Zhang's inequality 
$h(\Ci_n)\leq 2\deg\Ci_n\mu(\Ci_n)\le9(n+1)\log3\left(\frac{4n+3}{n}\right)$.

From formula \eqref{bhpunto}, if $P\in \Ci_n(\mathbb{Q})$ then 
\begin{equation}\label{esplicitoperCn}
\hat{h}(P)\leq 8.253\cdot 10^{38} (n+1)^3.
\end{equation}

Let us now write $P=([a]g,[b]g)$, with $g=(1,1)$  the generator of $E(\qe)$.
By the  definition of $\hat{h}$ on $E^2$ {(see \eqref{sumhcan})} and the properties of the N\'eron-Tate height, we have
\[
 \hat{h}(P)=(a^2+b^2)\hat{h}(g).
\]

{Now, by relations \eqref{BoundSilvermanAltezze}, \eqref{stima_zimmer} and $(x([a]g))^n=y([b]g)$ because $P$ is on the curve, we have}

\begin{align*}
 2na^2\hat{h}(g)&\leq n h(x([a]g))+2n\left(\frac{h(\Delta)}{12}+\frac{h_\infty(j)}{12}+1.07\right)\leq \\
 &=h(y([b]g))+5 n\leq h([b]g)+5 n\leq\\
 & \leq 3b^2\hat{h}(g)+6\log 2+5 n,
\end{align*}
and therefore
\[
 \left(\frac{2n}{3}+1\right)a^2\hat h(g)\leq \hat h(P)+2\log 2+\frac{5}{3}n.
\]
{Combining this with \eqref{esplicitoperCn} and the lower bound \eqref{hg} we obtain
\begin{equation*}
|a|\leq 7.037\cdot 10^{19}(n+1).
\end{equation*}
}

Using again \eqref{BoundSilvermanAltezze} and \eqref{stima_zimmer} as before, we get that
\begin{equation*}
2b^2\hat{h}(g)\leq 3n a^2 \hat{h}(g)+7\log 2+5
\end{equation*}
and from this and \eqref{hg} we get
\[|b|\leq \left(\frac{3na^2}{2}+14\log 2+10\right)^{\frac{1}{2}}.\qedhere\]
\end{proof}

\section{Estimates for degrees of maps}\label{section3}

The central aim of this section  is to produce sharp bounds for the degree of algebraic subgroups of $E^N$.  

 For example, consider the algebraic subgroup of codimension 1 defined by the morphism   $(l_{1},\cdots,l_{N}): E^N\to E$ sending $(X_1,\cdots, X_N) \mapsto l_{1}X_1+\cdots+l_{N}X_N$, where $l_1,\ldots, l_N\in \mathrm{End}(E)$. The degree of this subgroup is equal to a constant times  $|(l_{1}, \cdots,l_{N})|^2$. To get explicit results we need to compute this constant and, to this purpose, we have to describe the sum of two points and the multiplication by an integer on $E$ and $E^N$ in terms of rational maps. We will examine these maps in detail and bound their degrees. 
 These  bounds will be used to prove Propositions \ref{main} and \ref{mainCM},  which are the core of our main theorem.

We {briefly} anticipate here what is needed to prove the above mentioned propositions.
We consider an  algebraic subgroup $H$ of $E^N$ of codimension $s$. It is defined by $s$ equations
\begin{align*}
L_1&(X_1,\dotsc,X_N)=0, \\
&\vdots\\ 
L_s&(X_1,\dotsc,X_N)=0
\end{align*}
{where $L_i(X_1,\ldots,X_N)=l_{i_1}X_1+\cdots+l_{i_N}X_N$ are  morphisms from $E^N$ to $E$; here  the  coefficients are endomorphisms $l_{i_j}\in \mathrm{End}(E)$, and they are expressed by certain rational functions; similarily the $+$ that  appears in this expression is the addition map in $E^N$, which is expressed by a rational function of the coordinates.

More precisely, if the $X_i$'s are all points on $E$ with {affine} coordinates $(x_i,y_i)$ in $\P_2$, then $L_i(\mathbf{X})$ are also points on $E$ with coordinates in $\P_2$ $(x(L_i(\mathbf{X})), y(L_i(\mathbf{X})))$ which are rational functions of the  $x_i$'s and $y_i$'s.

The purpose of this section is to study the rational functions $x(L_i(\mathbf{X}))$ and  to bound the sums of the partial degrees (not only the partial degrees) of their numerators and denominators. The reason is that, in the  proof of the main theorem, we will need to study the image of the  algebraic subgroup $H$ in $\P_{3^N-1}$ via the Segre embedding, and when studying the effect of the Segre embedding on the functions defining the embedded variety, it is the sum of the partial degrees that comes into play.}

To this aim, we proceed in the  following way: in Subsection \ref{secpolinomi} we give bounds for the sums of the partial degrees of the product and the sum of quotients of polynomials in several variables.
Then, in Subsection \ref{secmoltiplicazionem} we estimate the multiplication map on $E$.
In Subsection \ref{eq_sub}, we study  the sum of many points on an elliptic curve. Finally, we estimate the sums of the partial degrees of $L_i(X_1,\ldots,X_N)$. 

All the computations are carried  out for linear combinations of points with integral coefficients (which is to say, when $E$ {is non} CM). In Remark \ref{remarkCMgradi} we describe how to adapt this to the CM case.

\subsection{Estimates for degrees of rational functions}\label{secpolinomi}
 In this short paragraph we {recall} how to bound the sums of the partial degrees of products and sums of quotients of polynomials in the field of rational functions.

If $\frac{f_i}{g_i}$ are rational functions, with $f_i,g_i$ polynomials in several variables and coefficients in $\Z$, we denote by $d(f_i/g_i)$ the maximum of the sums of the partial degrees of both $f_i,g_i$. Then
\begin{align}
\label{stimaprodottofunzioni}\frac{f}{g}&=\prod_{i=1}^r\frac{f_i}{g_i}  & d(f/g)&\leq \sum_{i=1}^r d(f_i/g_i)\\
\label{stimasommafunzioni}\frac{f}{g}&=\sum_{i=1}^r\frac{f_i}{g_i} & d(f/g)&\leq \sum_{i=1}^r d(f_i/g_i) 
\end{align}
where $d(f/g)$ is the bound for the sum of partial degrees of the product  (in \eqref{stimaprodottofunzioni}) and of the sum (in \eqref{stimasommafunzioni}) of the $f_i/g_i$'s respectively.

\subsection{The multiplication by $m$}\label{secmoltiplicazionem}

Let $m$ be a positive integer and let $E$ be an elliptic curve. The aim of this subsection is to bound the sum of the partial degrees of the rational function giving the multiplication by $m$ on $E$.
If $P=(x,y)\in E$, then by \cite{SilvermanArithmeticEllipticCurves}, Ex~3.7, p.~105
we have that
\begin{equation*}
[m]P=\left(\frac{\phi_m(P)}{\psi_m^2(P)},\frac{\omega_m(P)}{\psi_m^3(P)}\right)
\end{equation*}
where $\phi_m, \psi_m, \omega_m\in \mathbb{Z}[A,B,x,y]$ are certain polynomials defined below.

The polynomial $\psi_m$ is defined inductively as follows: 
\begin{align*}
\psi_1&=1, \psi_2=2y,\\
\psi_3 &= 3x^4 + 6Ax^2 + 12Bx-A^2,\\
\psi_4 &= 4y(x^6 + 5Ax^4 + 20Bx^3- 5A^2x^2 -4ABx-8B^2-A^3),
\end{align*}
and for $m\geq 2$
\begin{align*}
\psi_{2m+1} &=\psi_{m+2}\psi_m^3-\psi_{m-1}\psi_{m+1}^3, \phantom{rrr}(m\geq 2)\\
2y\psi_{2m}&=\psi_m(\psi_{m+2}\psi_{m-1}^2-\psi_{m-2}\psi_{m+1}^2) \phantom{rrr}(m\geq 2).
\end{align*} 
The polynomials $\phi_m$ and $\omega_m$ are defined as:
\[\phi_m=x\psi_m^2-\psi_{m+1}\psi_{m-1}\]
\[4y\omega_m=\psi_{m+2}\psi_{m-1}^2-\psi_{m-2}\psi_{m+1}^2.\]
As Silverman points out, one can prove that $\psi_m,\phi_m,y^{-1}\omega_m$ (for $m$ odd) and $(2y)^{-1}\psi_m, \phi_m, \omega_m$ (for $m$ even)  are polynomials in $\mathbb{Z}[A,B,x,y^2]$. Hence, using the Weierstrass equation for $E$ to replace $y^2$, they can be treated as polynomials in $\mathbb{Z}[A,B,x]$.

Moreover, as polynomials in $x$ we have
\[\phi_m(x)=x^{m^2}+\text{ lower order terms,}\]
\[\psi_m^2(x)=m^2 x^{m^2-1}+\text{ lower order terms.}\]

\medskip

We {need} to find bounds for the degrees of the polynomials $\psi_m, \phi_m, \omega_m$.

The following bounds are obtained combining the above definitions in \cite{SilvermanArithmeticEllipticCurves} {and the expressions for $\phi_m$ and $\psi_m^2$:}
\begin{align*}
 d(\phi_m)&=m^2\\
 d(\psi_m^2)&= m^2-1\\
 d(\psi_m)&\leq \frac{m^2+1}{2}\\
 d(\psi_m^3)&\leq \frac{3m^2-1}{2}\\
 d(\omega_m)&\leq  \frac{3}{2}(m^2+1).
\end{align*}

Using the above bounds and the formula for the coordinates of $[m]P$, we see that the sum of the partial degrees of the polynomials $\phi_m, \psi_m^2,\omega_m,\psi_m^3$, which are numerators and denominators of the rational functions given by the coordinates of $[m]P$, are bounded by $\frac{3}{2}(m^2+1)$.

\bigskip

\subsection{Estimates for linear maps}\label{eq_sub} 
We now look first at the functions giving the sum of two, and then many, points on an elliptic curve. Our aim is to obtain explicit bounds on the sum of the partial degrees of the rational function expressing a linear combination of $N$ points in $E$. 

Then we will study equations  defining  algebraic subgroups of $E^N$, obtained equating to zero linear combinations of $N$ variables with coefficients in $\mathrm{End}(E)$. 
Evaluating the functions at points in $E$ whose coordinates are themselves rational functions, we will bound the sum of partial degrees of these linear combinations, viewed as rational functions in {the} new coordinates.

\subsubsection{Estimates for the addition map}
We  consider an elliptic curve $E$ embedded in $\mathbb{P}_2$ via its Weierstrass equation.

Let $P_1=(x_1,y_1)$ and $P_2=(x_2,y_2)$ be points on $E$ and $P_3=(x_3,y_3)=P_1\oplus P_2$ be their sum.

 If $x_1\neq x_2$, from \cite{SilvermanArithmeticEllipticCurves}, Chap. 3, setting
\begin{align*}
\lambda&=\frac{y_2-y_1}{x_2-x_1}\\
\nu&=\frac{y_1x_2-y_2x_1}{x_2-x_1}
\end{align*}
we have that
\begin{equation}\label{somma2punti}
\begin{split}
x_3&=\lambda^2-x_1-x_2\\
y_3&=-\lambda x_3 -\nu.
\end{split}
\end{equation}

So $x_3$ and $y_3$ are rational functions of the coordinates of $P_1$ and $P_2$, and we  now want to control the sum of their partial degrees.

 Using \eqref{somma2punti}, if $(x_1,y_1)$ and $(x_2,y_2)$  already have coordinates given by  certain other rational functions,  whose sums of the partial degrees  are bounded  respectively by $d_1,d_2$, then the sums of the partial degrees, in the variables of $x_1,y_1,x_2,y_2$, of the functions $x_3,y_3$ are given by:
\begin{align*}
d(x_3)&\leq 5(d_1+d_2)\\
d(y_3)&\leq 12(d_1+d_2).
\end{align*}

 If instead $x_1=x_2$ and $y_1=y_2$, then setting 
\begin{equation*}
 \lambda=\frac{3x_1^2+A}{2y_1}
\end{equation*}
we have that
\begin{align*}
 x_3&=\lambda^2-2x_1\\
 y_3&=y_1+\lambda(x_3-x_1).
\end{align*}
Again applying \eqref{stimaprodottofunzioni}, \eqref{stimasommafunzioni}, if $P_1=P_2$ have coordinates given by rational functions in certain variables with sums of the partial degrees bounded by $d_1$, then $x_3$ and $y_3$ are rational functions, in the same variables, with sums of the partial degrees bounded as

\begin{align*}
d(x_3)&\leq 7d_1 \\
d(y_3)&\leq 11d_1 .
\end{align*}

If $x_1=x_2$ and $y_1=-y_2$, {then} the two points are opposite and the sum is the zero of the elliptic curve.

Comparing the bounds obtained in the three cases, one checks that in any case
\begin{align}\label{gradosomma2punti}
\max(d(x_3),d(y_3))&\leq 12(d_1+d_2)
\end{align}
holds.
This computation was carried out for the sum of two points. We now iterate it $M-1$ times to obtain bounds for the rational function expressing the sum of $M$ points.
It follows by induction from \eqref{gradosomma2punti} that

\begin{align}
\label{sommaNpuntigrado}d&\leq 12^{M-1} d_1 + \sum_{i=2}^{M} 12^{M-i+1}d_i\leq 12^{M-1} \sum_{i=1}^{M}d_i,
\end{align}
where $d_i$ is a bound for the sum of the partial degrees of the $x$ and $y$ coordinates of the $i$-th point.

\subsubsection{Estimates for group morphisms}\label{sec62}
Let us consider the {morphism $L:E^N\rightarrow E$ defined as} 
\begin{equation*}
L(\mathbf{X})=l_1X_1+\dotsb+l_NX_N
\end{equation*}  
{where $l_i\in \mathbb{Z}$ and $X_i=(x_i,y_i)$ is in the $i$-th factor of $E^N$.} Then $L(\mathbf{X})$ is also a point on $E$ with coordinates $(x(L(\mathbf{X})), y(L(\mathbf{X})))$. {By} the considerations above, {these} are rational functions in the coordinates $(x_i,y_i)$ of all $X_i$'s. We want to combine the results from the previous subsections to bound the sum of the partial degrees of the rational function $x(L(\mathbf{X}))$.

Let us set $d(L)=d(x(L(\mathbf{X})))$ to be the sum of the partial degrees in the numerator and denominator of $x(L(\mathbf{X}))$.

Now combining inequality \eqref{sommaNpuntigrado} with the bounds from Subsection \ref{secmoltiplicazionem} we obtain 
\begin{equation}\label{combinazioneNpuntigrado}
d(L)\leq\frac{3}{2}12^{N-1}\left(N+\sum_{i=1}^N|l_i|^2\right).
\end{equation}

\begin{remark}\label{remarkCMgradi}
The content of this section holds analogously  when $E$ is CM. In this case $\mathrm{End}(E)=\Z+\tau\Z$ for some imaginary quadratic integer $\tau$. 

 For a point $P=(x,y)$ we denote $\tau(P)=(x_\tau,y_\tau)$. Then $x_\tau$ and $y_\tau$ are rational functions of $x$ and $y$ and we  let $d_{CM}(\tau)$  be the sum of their partial degrees. Writing $\tau=\sqrt{-d}$ for some non-square  positive integer $d$ we have
\[d_{CM}(\tau)\leq 2d.\]
Since every element $l_i\in \mathrm{End}(E)$ can be written as $l_i=r_i+\tau s_i$, where $r,s\in \mathbb{Z}$, we can write 
\[
 L(\mathbf{X})=r_1X_1+\dotsb+r_NX_N+\tau s_1X_{N+1}+\dotsb+\tau s_NX_{2N}
\]
and,  using the above results, we can effectively compute a bound
\[
d(L)\leq D(E,N,\tau)\sum_{i=1}^N|l_i|^2\qquad \text{for $l_i\in \mathrm{End}(E)$ not all $0$,}
\]
corresponding to the bound \eqref{combinazioneNpuntigrado}. However we omit here the explicit computations. 
\end{remark}

\section{Conclusion}\label{proofmain}
{In this section we prove Proposition \ref{main} (explicit for non CM varieties)  and Proposition \ref{mainCM} (effective for CM varieties), 
which are  the core of the proofs of our main results. Essentially the {propositions} state that:} if a point $P\in E^N$ belongs to an  algebraic subgroup $B$  then we can construct a translate  $H+P\subseteq E^N$  of controlled degree and height.

To prove {these results}, we use the bounds for the height of Subsection \ref{SezioneAltezze} and for the degree of Section \ref{section3}. Then we use the geometry of {numbers}, to construct the  algebraic subgroup $H$.

 Before moving on with the main proof, we need a short section in linear algebra.
\subsection{A lemma on adjugate matrices}\label{secMatrices}
Let $A$ be a $n\times n$ matrix with complex coefficients. Let $a_i\in\C^n$ be the rows of $A$.
\begin{D}\label{defU*} The \emph{adjugate matrix of $A$}, denoted $A^*$, is  the transpose of the matrix $((-1)^{i+j}\det M_{ij})_{ij}$, where $M_{ij}$ is the $(n-1)\times(n-1)$ minor obtained from $A$ after deleting the $i$-th row and the $j$-th column.
\end{D}
The adjugate matrix has the property that
\[
AA^*=A^*A=(\det A)\mathrm{Id}
\]
and its entries are bounded as it follows:
\begin{lem}\label{lemmaadjugate}
Let $A\in M_{n\times n}(\C)$  be the matrix with rows $a_1,\dotsc,a_n\in\C^n$.
Then every entry in the $i$-th column of $A^*$ has absolute value bounded by \[\frac{\abs{a_1}\dotsm \abs{a_n}}{\abs{a_i}}.\]
\end{lem}
\begin{proof}
Applying Hadamard's inequality  to $M_{ij}$, and denoting by $a'_i\in\C^{n-1}$ the vector obtained from $a_i$ after deleting the $j$-th entry, we have that
\[
\abs{\det M_{ij}}\leq \prod_{k\neq i}\abs{a'_k}\leq \prod_{k\neq i}\abs{a_k}= \frac{\abs{a_1}\dotsm \abs{a_n}}{\abs{a_i}}.
\]
The thesis follows multiplying by $(-1)^{i+j}$ and transposing.
\end{proof}

\subsection{A bound for the height}\label{subsec-alttrasl}
To estimate the height of $H+P$ we use an argument based on linear algebra, and some bounds on heights from Subsection \ref{SezioneAltezze}.

\bigskip

Let $H$ be a component of the  algebraic subgroup defined by the $s\times N$ matrix with rows $u_1,\dotsc,u_s\in \Z^N$.
Let $\Lambda\subseteq\R^N$ be the associated lattice, and $\Lambda^\perp$ its orthogonal lattice.
{Let $u_{s+1},\dotsc,u_N$ be a basis of $\Lambda^\perp$ such that $\abs{u_{s+1}},\dotsc,\abs{u_N}$ are the successive minima of $\Lambda^\perp$, as defined in Subsection \ref{prelim-S}.}

The $(N-s)\times N$ matrix with rows $u_{s+1},\dotsc,u_N$ defines an  algebraic subgroup $H^\perp$, and for any point $P\in E^N$ there are  two points $P_0\in H$, $P^\perp\in H^\perp$, unique up to torsion points in $H\cap H^\perp$, such that $P=P_0+P^\perp$.

Let $U$ be the $N\times N$ matrix with rows $u_1,\dotsc,u_N$, and let $\Delta$ be its determinant.

Notice that
\[
 |\Delta|=\det\Lambda \cdot \det \Lambda^\perp
\]
because $\Lambda$ and $\Lambda^\perp$ are orthogonal.

We remark that $u_i(P_0)=0$ for all $i=1,\dotsc,s$, because $P_0\in H$, and $u_i(P^\perp)=0$ for all $i=s+1,\dotsc,N$ because $P^\perp\in H^\perp$.

Therefore

\begin{equation*}
UP^\perp=
\left(
\begin{array}{c}
u_1(P^\perp)\\
\vdots\\
u_s(P^\perp)\\
0\\
\vdots\\
0
\end{array}
\right)
=
\left(
\begin{array}{c}
u_1(P_0+P^\perp)\\
\vdots\\
u_s(P_0+P^\perp)\\
0\\
\vdots\\
0
\end{array}
\right)
=
\left(
\begin{array}{c}
u_1(P)\\
\vdots\\
u_s(P)\\
0\\
\vdots\\
0
\end{array}
\right),
\end{equation*}
hence
\begin{equation*}
[\Delta]P^\perp=U^*UP^\perp=U^*\left(
\begin{array}{c}
u_1(P)\\
\vdots\\
u_s(P)\\
0\\
\vdots\\
0
\end{array}
\right)
\end{equation*}
where $U^*$ is the adjugate matrix of $U$ from Definition \ref{defU*}.

Computing canonical heights and applying Lemma \ref{lemmaadjugate} yields
\begin{equation*}
\abs{\Delta}^2\hat h(P^\perp)=\hat h([\Delta]P^\perp)\leq N\abs{u_1}^2\dotsm \abs{u_N}^2\sum_{i=1}^s\frac{\hat h(u_i(P))}{\abs{u_i}^2}.
\end{equation*}

Recall inequality \eqref{mu2mu^}, which gives
 \begin{equation*}
\mu(H+P)\leq 3 \hat\mu(H+P) + \csedici(E,N),
\end{equation*}
where  $\csedici(E,N)$ was defined as
\begin{equation*}
\csedici(E,N)=N(3h_{\mathcal W}(E)+6\log 2+\frac{1}{2}\log 3).
\end{equation*}
{By \cite{preprintPhilippon} we know that}
\[
\hat\mu(H+P)=\hat h(P^\perp)
\]
and therefore, by Zhang's inequality
\begin{align}
\notag h(H+P)&\leq (N-s+1)(\deg H)\mu(H+P)\leq\\
\notag &\leq (N-s+1)(\deg H)(3 \hat\mu(H+P) +\csedici(E,N))\leq\\
\notag &\leq (N-s+1)(\deg H)(3\hat h(P^\perp)  +\csedici(E,N))\leq\\
\label{ultimarigacost}&\leq  (N-s+1)\deg H \left(\frac{3N}{\abs{\Delta}^2}\abs{u_1}^2\dotsm \abs{u_N}^2\sum_{i=1}^s\frac{\hat h(u_i(P))}{\abs{u_i}^2}+\csedici(E,N)\right).
\end{align}

By \eqref{gradorighe}  we get
\[
 \deg H\leq \ccinque(N,s)\prod_{i=1}^s |u_i|^2,
\]
by \eqref{gradodet}  we obtain
\[
\frac{\deg H}{(\det\Lambda)^2}\leq \ccinque(N,s)\frac{4^s}{\omega_s^2},
\]
and by {Minkowski's second theorem}
\[
 \frac{\prod_{i=s+1}^N |u_i|}{\det\Lambda^\perp}\leq \frac{2^{N-s}}{\omega_{N-s}}.
\]
Plugging these inequalities in \eqref{ultimarigacost} we obtain

\[
 h(H+P)\leq \frac{3N(N-s+1)4^N}{(\omega_{N-s}\omega_{s})^2}\ccinque(N,s)\prod_{i=1}^s |u_i|^2 \sum_{i=1}^s\frac{\hat h(u_i(P))}{\abs{u_i}^2}+\csedici(E,N)(N-s+1)\ccinque(N,s)\prod_{i=1}^s |u_i|^2.
\]

So, we have proved the following proposition
\begin{propo}\label{stimaaltezzanormalizzatatraslato}
Let $E$ be a {non CM} elliptic curve. Let $H\subseteq E^N$ be a component of the  algebraic subgroup associated with an $s\times N$ matrix with rows $u_1,\dotsc,u_s\in \Z^N$. Then 
\begin{equation*}
h(H+P)\leq \frac{3N(N-s+1)4^N}{(\omega_{N-s}\omega_{s})^2}\ccinque(N,s)\prod_{i=1}^s |u_i|^2 \sum_{i=1}^s\frac{\hat h(u_i(P))}{\abs{u_i}^2}+\cquattro(E,N,s)\prod_{i=1}^s |u_i|^2,
\end{equation*}
where $\omega_n$ is defined at \eqref{valomega},
\begin{equation*}
 \ccinque(N,s)=3^N N! \left( \frac{3}{2}(N+1)12^{N-1}\right)^s
\end{equation*}
and
\begin{equation*}
 \cquattro(E,N,s)=N(N-s+1)\ccinque(N,s)(3h_{\mathcal W}(E)+6\log 2+\frac{1}{2}\log 3).
\end{equation*}
\end{propo}

\subsection{A bound for the degree}\label{sec-c6}
Here we use an inductive geometric construction to bound the degree of a translate. 

We first consider an  algebraic subgroup given by a single equation in $E^N$.
Then  we apply the Segre embedding and see  this subgroup as a subvariety of $\P_{3^N-1}$. In doing this we must be careful in selecting irreducible components. Finally we apply inductively B\'ezout's theorem for the case of several equations.

\medskip

Let $U$ be a matrix in $\mathrm{Mat}_{s\times N}(\mathbb{Z})$ with rows $u_1,\dotsc,u_s\in \Z^N$ and let $H\subseteq E^N$ be an irreducible component of the algebraic subgroup associated with the matrix $U$ (see Subsection \ref{prelim-S}).

If $X_1=(x_1,y_1),\ldots, X_N=(x_N,y_N)$ are points on $E$ and  $v=(v_1,\ldots,v_N)\in \mathbb{Z}^N$ is a vector, we denote $v(\mathbf{X})=v_1 X_1+\ldots+v_N X_N$.

As remarked in the previous section, $v(\mathbf{X})=(x(v(\mathbf{X})), y(v(\mathbf{X})))$ is a point in $E$ and $x(v(\mathbf{X}))$ is a rational function of the $x_i,y_i$'s.

Let now $P=(P_1,\ldots,P_N)\in E^N$ be a point.
Take the $k$-th row $u_k\in \Z^N$ of $U$  and consider the equation
\[x(u_k(\mathbf{X}))=x(u_k(P))\]
with $\mathbf{X}=(X_1,\ldots,X_N)\in E^N$ as before.

 Clearing out the denominators the previous equation can be written as 
\[f_{u_k,P}(x_1,y_1,\dotsc,x_N,y_N)=0,\]
where $f_{u_k,P}$ is a polynomial of degree bounded by $d(u_k)$ (see formula \eqref{sommaNpuntigrado}). 

This polynomial defines a variety in $\P_2^N$. Applying the Segre embedding, we want to study this variety as a subvariety of $\P_{3^N-1}$.

The Segre embedding induces a morphism between the fields of rational functions, whose effect on the polynomials in the variables $(x_1,y_1,\dotsc,x_N,y_N)$ is simply to replace any monomial in the variables of $\P_2^N$ with another monomial in the new variables, without changing the coefficients; the total degree in the new variables is the sum of the partial degrees in the old ones.

Recall that in Subsection \ref{sezioneSegre} we defined $X(E,N)$ as the image of $E^N$ in $\P_{3^N-1}$.

\medskip

Denote by  $Y'_{k}\subseteq \P_{3^N-1}$ the zero-set of the polynomial  $f_{u_k,P}(x_1,y_1,\dotsc,x_N,y_N)$ after embedding $\P_2^N$ in $\P_{3^N-1}$. 

Now consider an irreducible component of the translate in $E^N$  defined by \[u_k(\mathbf{X})=u_k(P)\] and denote by $Y_k$ its image  in $\P_{3^N-1}$. We want to obtain bounds for the degree of the hypersurfaces $Y_k$.

Notice that 
\[Y_k\subseteq Y'_k\cap X(E,N)\] and it is a component.
This is because setting the first coordinate of $u_k(\mathbf{X})$ equal to $x(u_k(P))$ defines two cosets, $u_k(\mathbf{X})=u_k(P)$ and $u_k(\mathbf{X})=-u_k(P)$.

By B\'ezout's theorem
\begin{align*}
 \deg Y_k\leq \deg X(E,N)\deg{ Y'_k}\leq \cuno\frac{3}{2}12^{N-1}\left(N+|u_k|^2\right),
\end{align*}
where the last inequality follows from formula \eqref{combinazioneNpuntigrado} and the definition of $\cuno$ in Lemma \ref{defic2}.

\smallskip

In a similar way, considering all the rows we get
\begin{align}\label{stgrad}
 \notag\deg (H+P)&\leq  \deg X(E,N)\deg Y'_1\dotsm\deg Y'_s\leq \cuno\left(\frac{3}{2}12^{N-1}\right)^s\prod_{i=1}^s(|u_i|^2+N)\leq\\
 &\leq \cuno\left(\frac{3}{2}(N+1)12^{N-1}\right)^s\prod_{i=1}^s|u_i|^2
\end{align}
{where we recall that, from relation \eqref{valc4}, $\cuno=3^N N!$.}

\begin{remark}
Clearly the degree of $H+P$ is equal to the degree of $H$ and does not depend on $P$. {Thus we can deduce} the value of the constant $\ccinque(N,r)$ of {Subsection} \ref{prelim-S}, {formula \eqref{gradorighe},} when the elliptic curve $E$ {is non} CM; by the previous inequality, we may take 
\begin{equation}\label{valc6}\ccinque(N,r)=3^N N!\left(\frac{3}{2}(N+1)12^{N-1}\right)^r.\end{equation}
\end{remark}

\subsection{Geometry of numbers}\label{subsec-geomnum}
In this subsection, inspired by the work of Habegger \cite{hab}, we give two lemmas  based on tools from the geometry of numbers which are used to define $H$.
We have the following version for powers of elliptic curves of Lemma 1 in \cite{hab}.

\begin{lem}\label{LemmaHab1}
Let $1\leq m \leq N$  be integers and let $P=(P_1,\dotsc,P_N)\in B\subseteq E^N$, where $B$ is a torsion variety of dimension $\leq m$ and $E$ {is non} CM.

There exist linear forms $L_1,\dotsc,L_m\in\R[X_1,\dotsc,X_N]$ such that $|L_j|\leq 1\quad \forall j$, where $|L_j|$ is the euclidean norm of the vector of the coefficients of $L_j$, and 
\begin{equation*}
\hat h(t_1P_1+\dotsb+t_NP_N)\leq \cdiciassette(N,m) \max_{1\leq j\leq m}\{|L_j(\mathbf{t})|^2\}\hat h (P)
\end{equation*}
for all $\mathbf{t}=(t_1,\ldots,t_N)\in\Z^N$. The constant $\cdiciassette(N,m)$ is given by 
\begin{equation*}
\cdiciassette(N,m)=\frac{m^3(m!)^4N}{4^{m-1}}.
\end{equation*}
\end{lem}
\begin{proof}
The points $P_i$ lie in a finitely generated subgroup of $E$ of rank at most $m$.

By \cite{ioannali}, Lemma 3, there are elements $g_1,\dotsc,g_m\in E$, and  torsion points $\zeta_1,\dotsc,\zeta_N\in E$, such that
\begin{equation*}
P_i=\zeta_i+ v_{i1}g_1\dotsb+ v_{im}g_m\quad\text{ for }i=1,\dotsc,N\text{ and some }v_{ij}\in\Z
\end{equation*}
and
\begin{equation*}
\hat h(b_1g_1+\dotsb+b_mg_m)\geq \frac{2^{2m-2}}{m^2(m!)^4}\max_{1\leq i\leq m}\{|b_i|^2\hat h(g_i)\}\quad \forall\mathbf{b}\in\Z^m.
\end{equation*}

Let $A=\max_{i,j}\{|v_{ij}|^2\hat h(g_j)\}$ and define
\begin{align*}
\tilde{L}_j&=v_{1j}X_1+\dotsb+v_{Nj}X_N &j&=1,\dotsc,m\\
L_j&=\left(\frac{\hat h(g_j)}{NA}\right)^\frac{1}{2}\tilde{L}_j &j&=1,\dotsc,m.
\end{align*}
Notice that we can assume $A>0$, otherwise the point $P$ would be a torsion point, and the thesis of the  lemma would be trivially true. Notice also that $|L_j|\leq 1$.

With these definitions, for every $\mathbf{t}\in\Z^N$ we have that 
\begin{equation*}
t_1P_1+\dotsb+t_NP_N=\xi+\sum_{i=1}^m \tilde{L}_j(\mathbf{t})g_j
\end{equation*}
where $\xi$ is a  torsion point.
Therefore
\begin{align}
\notag \hat h(t_1P_1+\dotsb+t_NP_N)=\hat h\left(\sum_{j=1}^m \tilde{L}_j(\mathbf{t})g_j\right)\leq\sum_{j=1}^m|\tilde{L}_j(\mathbf{t})|^2 \hat h(g_j)=\\
\label{ineqlemma}= NA\sum_{j=1}^m|{L}_j(\mathbf{t})|^2\leq mNA\max_{1\leq j\leq m}\{|{L}_j(\mathbf{t})|^2\}.
\end{align}

If $i_0,j_0$ are the indices for which the maximum is attained in the definition of $A$, then
\begin{equation*}
\frac{2^{2m-2}}{m^2(m!)^4}A=\frac{2^{2m-2}}{m^2(m!)^4}|v_{i_0 j_0}|^2\hat h(g_{j_0})\leq \hat h(P_{i_0})\leq \hat h(P).
\end{equation*}
Combining this with inequality \eqref{ineqlemma}, we get the thesis of the lemma.
\end{proof}
We now recall the following lemma  of Habegger (\cite{hab}, Lemma 3),  obtained applying {Minkowski's second theorem}.

\begin{lem}\label{LemmaHab3}
Let $1\leq m \leq N$ and let $L_1,\dotsc, L_m\in\R[X_1,\dotsc,X_N]$ be linear forms with $\abs{L_j}\leq 1\quad\forall j$. If $T\geq 1$, then for any integer $s$ with $1\leq s\leq n$ there exist linearly independent $u_1,\dotsc, u_s\in\Z^N$ such that $\abs{u_1}\dotsm\abs{u_s}\leq T$ and
\[
\abs{u_1}\dotsm \abs{u_s}\frac{L_j(u_k)}{\abs{u_k}}\leq \cdiciotto(N,m)T^{1-\frac{N}{ms}}
\]
for $1\leq j\leq m$ and $1\leq k\leq s$ and 
\begin{equation*}
\cdiciotto(N,m)=\binom{m+N}{N}^{1/2}\frac{4^N}{\omega_N}.
\end{equation*}
The value of $\cdiciotto(N,m)$ follows from formulae (27) and (28) in the proof of \cite{hab}, Lemma 2.
\end{lem}

\subsection{The proofs  of Propositions \ref{main} and \ref{mainCM}.}

\subsubsection{The non-CM case}

\begin{proof}[Proof of  Proposition \ref{main}]
By Lemma \ref{LemmaHab1}, and Lemma \ref{LemmaHab3} applied to $\sqrt{T}$, there are linearly independent vectors $u_1,\dotsc,u_s\in\Z^N$ such that \begin{equation}\label{bgrad}|u_1|^2\dotsm|u_s|^2\leq T\end{equation} and 
\begin{equation*}
\left(\frac{|u_1|\dotsm|u_s|}{|u_k|}\right)^2 \hat h(u_k(P))\leq \cdiciassette(N,m) \cdiciotto(N,m)^2 T^{1-\frac{N}{ms}}\hat h(P).
\end{equation*}

If we consider the  algebraic subgroup defined by equations $u_k(\mathbf{X})=0$, and call $H$ the irreducible component containing $0$, {then, combining \eqref{bgrad} and \eqref{stgrad}, its degree is bounded as}
\[\deg(H+P)\leq \ccinque(N,s) T\] 
{and we can use Proposition \ref{stimaaltezzanormalizzatatraslato}} to bound the height of $H+P$ as
\begin{equation*}
h(H+P)\leq \frac{3sN(N-s+1)4^N}{(\omega_s\omega_{N-s})^2}\ccinque(N,s)\cdiciassette (N,m)\cdiciotto(N,m)^2T^{1-\frac{N}{ms}}\hat h(P)+\cquattro(E,N,s)T.\qedhere
\end{equation*}
\end{proof}

\subsubsection{The CM case}\label{subsecCM}
\begin{proof}[Proof of  Proposition \ref{mainCM}]
If the elliptic curve  is CM, one can still apply the arguments of  Section \ref{proofmain}. More precisely, the geometric arguments of Subsection \ref{sec-c6} do not depend on $E$ having CM, one only needs to replace the application of formula \eqref{combinazioneNpuntigrado} with the corresponding bound for the CM case, as discussed in Remark \ref{remarkCMgradi}. 

{The argument of Subsection \ref{subsec-alttrasl} assumes that $\emor (E) = \Z$ because it uses the formulae from Subsection \ref{prelim-S}, which are straightforward consequences of the second Minkowski's theorem, here stated in the classical form for a lattice in $\R^N$. Analogous inequalities may be derived, when $\emor (E)$ is an order in an imaginary quadratic number field, from more general reformulations of Minkowski's theorem, such as Theorem 3 of \cite{BomVal}. The linear algebra used in Subsection \ref{subsec-alttrasl} remains the same if the entries of the matrix lie in $\C$ instead of $\R$. 
}

Of the two lemmas in Subsection \ref{subsec-geomnum}, Lemma \ref{LemmaHab3} holds regardless of whether $E$ is CM, while in Lemma \ref{LemmaHab1} it is necessary  to replace \cite{ioannali}, Lemma 3 with the  Proposition 2 of \cite{ioannali}.
This proves  Proposition \ref{mainCM}.
\end{proof}

\section*{Acknowledgments} We thank P. Habegger for useful discussions and for bringing to our attention the construction in \cite{hab}.  We thank the referee for his remarks, that gave us the opportunity to improve our paper.  We warmly thank the FNS for the financial support and for optimal working conditions.

\def\cprime{$'$}
\providecommand{\bysame}{\leavevmode\hbox to3em{\hrulefill}\thinspace}
\providecommand{\MR}{\relax\ifhmode\unskip\space\fi MR }
\providecommand{\MRhref}[2]{%
  \href{http://www.ams.org/mathscinet-getitem?mr=#1}{#2}
}
\providecommand{\href}[2]{#2}

\section*{}
\noindent Sara Checcoli:
 Institut Fourier, 
100 rue des Maths,
BP74 38402 Saint-Martin-d'H\`eres Cedex, France.
email: sara.checcoli@ujf-grenoble.fr
\medskip\\
Francesco Veneziano:
Mathematisches Institut,
Universit\"{a}t Basel,
Spiegelgasse 1,
CH-4051 Basel,
Switzerland.
email: francesco.veneziano@unibas.ch
\medskip\\
Evelina Viada:
Mathematisches Institut, 
Georg-August Universit\"{a}t G\"{o}ttingen,
Bunsenstra\ss e 3-5,
D-37073 G\"{o}ttingen, Germany.
email: viada@uni-math.gwdg.de

\end{document}